\documentclass[11pt,a4paper]{amsart}
\usepackage{amssymb}
\usepackage[a4paper,twoside,centering]{geometry}
\usepackage[arrow,matrix,curve]{xy}

\title{Finite-dimensional algebras are $(m>2)$-Calabi-Yau-tilted}
\author{Sefi Ladkani}
\thanks{This work was partially supported by the Center for Absorption
in Science, Ministry of Aliyah and Immigrant Absorption, State of Israel.}
\address{Department of Mathematics, University of Haifa, Mount Carmel,
Haifa 31905, Israel}
\email{ladkani.math@gmail.com}

\DeclareMathOperator{\add}{add}
\DeclareMathOperator{\End}{End}
\DeclareMathOperator{\Ext}{Ext}
\DeclareMathOperator{\fd}{fd}
\DeclareMathOperator{\gldim}{gldim}
\DeclareMathOperator{\hh}{H}
\DeclareMathOperator{\Hom}{Hom}
\DeclareMathOperator{\id}{id}

\DeclareMathOperator{\pd}{pd}
\DeclareMathOperator{\per}{per}
\DeclareMathOperator{\RHom}{\mathbf{R}Hom}

\newcommand{\wb}{\overline}
\newcommand{\wt}{\widetilde}

\newcommand{\eps}{\varepsilon}
\newcommand{\cC}{\mathcal{C}}
\newcommand{\cD}{\mathcal{D}}
\newcommand{\cF}{\mathcal{F}}
\newcommand{\fr}{\mathfrak{r}}
\newcommand{\cY}{\mathcal{Y}}
\newcommand{\bZ}{\mathbb{Z}}

\newtheorem{theorem}{Theorem}[section]
\newtheorem{lemma}[theorem]{Lemma}
\newtheorem{prop}[theorem]{Proposition}
\newtheorem{cor}[theorem]{Corollary}

\theoremstyle{definition}
\newtheorem{defn}[theorem]{Definition}
\newtheorem{remark}[theorem]{Remark}
\newtheorem{example}[theorem]{Example}
\newtheorem{constr}[theorem]{Construction}

\numberwithin{equation}{section}

\begin{document}

\begin{abstract}
We observe that over an algebraically closed field, any finite-dimensional
algebra is the endomorphism algebra of an $m$-cluster-tilting object in a
triangulated $m$-Calabi-Yau category, where $m$ is any integer greater than
$2$.
\end{abstract}

\maketitle

\section{Introduction}

Cluster categories were introduced by Buan, Marsh, Reineke, Reiten
and Todorov~\cite{BMRRT06} as a means to model the combinatorics of
cluster algebras with acyclic skew-symmetric exchange matrices within
the framework of quiver
representations (see also the work of Caldero, Chapoton and
Schiffler~\cite{CCS06} for the case of $A_n$ quivers).

The cluster category of an acyclic quiver is the orbit category of
the bounded derived category of its path algebra with
respect to the autoequivalence $F=\nu \Sigma^{-2}$ where $\nu$ denotes
the Serre functor and $\Sigma$ is the suspension functor.
By a result of Keller~\cite{Keller05}, the cluster category is a triangulated
2-Calabi-Yau category. Particular role is played by the 
2-cluster-tilting objects within this category
(the precise definitions will be given in Section~\ref{ssec:notat} below)
which model the clusters in the corresponding cluster algebra.

As already shown in~\cite{Keller05}, by replacing $\Sigma^{-2}$ by
$\Sigma^{-m}$ for $m>2$ and considering the orbit category with respect to the
autoequivalence $\nu \Sigma^{-m}$, one gets an $m$-Calabi-Yau triangulated
category with $m$-cluster-tilting objects. The categories obtained in this way,
called \emph{$m$-cluster categories}, were the subject of many investigations,
see~\cite{BaurMarsh08,Thomas07,Wraalsen09,ZhouZhu09}.

The endomorphism algebras of 2-cluster-tilting objects in cluster categories
are known as \emph{cluster-tilted algebras} and they possess many remarkable
representation-theoretic and homological
properties~\cite{ABS08,BMR06,BMR07,KellerReiten07}.
More generally, consider an \emph{$m$-Calabi-Yau-tilted algebra}, i.e.\
an algebra $A=\End_{\cC}(T)$ where $\cC$ is a $K$-linear,
triangulated, $\Hom$-finite, $m$-Calabi-Yau category over a field $K$ and $T$
is an $m$-cluster-tilting object in $\cC$ for some positive integer $m$.
Keller and Reiten have shown the following results in the case $m=2$
(for the first two points, see Sections~2 and~3 of~\cite{KellerReiten07}
and for the third one, see \cite[\S2]{KellerReiten08}):
\begin{itemize}
\item
$A$ is Gorenstein of dimension at most $m-1$ (i.e.\ $\id_A A \leq m-1$ and
$\pd_A DA \leq m-1$, where $D(-)=\Hom_K(-,K)$);

\item
The stable category of Cohen-Macaulay $A$-modules
is $(m+1)$-Calabi-Yau;

\item
If $K$ is algebraically closed, $\cC$ is algebraic and $A \cong KQ$ for
an acyclic quiver $Q$, then $\cC$ is triangle equivalent to the $m$-cluster
category of~$Q$.
\end{itemize}

In addition, they have shown that these results hold also in the case $m>2$
provided that one imposes an additional condition on the $m$-cluster-tilting
object $T$ stated in terms of the vanishing of some of its negative
extensions, namely
\begin{equation} \label{e:vosnex}
\tag{$\star$}
\Hom_{\cC}(T, \Sigma^{-i} T)=0 \text{ for any $0 < i < m-1$},
\end{equation}
see \cite[\S4]{KellerReiten08}.
Note that by the $m$-Calabi-Yau property of $\cC$ this condition is equivalent
to the vanishing of the positive extensions $\Hom_{\cC}(T, \Sigma^i T)$ for all
$m < i < 2m-1$, whereas these extensions for $0 < i < m$ always vanish since
$T$ is $m$-cluster-tilting.

In the terminology of~\cite{Beligiannis15}, the condition~\eqref{e:vosnex}
means that $T$ is $(m-2)$-corigid, and
in~\cite[Theorem~B]{Beligiannis15} Beligiannis presents a more refined result
connecting the corigidity property of an $m$-cluster-tilting object 
with the Gorenstein property of its endomorphism algebra and the Calabi-Yau
property of its stable category of Cohen-Macaulay modules.
We note also that a condition analogous to~\eqref{e:vosnex}, stated for
$m$-cluster-tilting subcategories inside bounded derived categories of modules
and abbreviated ``vosnex'', appears in the works of
Amiot-Oppermann~\cite[Definition~4.9]{AmiotOppermann14a} and
Iyama-Oppermann~\cite[Notation~3.5]{IyamaOppermann13}.

Whereas for $m=2$ the condition~\eqref{e:vosnex} is empty and hence
automatically holds, this is no longer the case for $m \geq 3$.
In particular, there are examples of triangulated $m$-Calabi-Yau categories
$\cC$ and $m$-cluster-tilting objects $T \in \cC$ for which the vosnex
condition~\eqref{e:vosnex} does not hold and moreover the endomorphism
algebra $\End_{\cC}(T)$ is a path algebra of a connected acyclic quiver, see
Iyama and Yoshino \cite[Theorem~9.3]{IyamaYoshino08} for $m=3$ and
\cite[Theorem~10.2]{IyamaYoshino08} for $m$ odd, and also
\cite[Example~4.3]{KellerReiten08}.
Another example, where the algebra $\End_{\cC}(T)$ is not Gorenstein, is given
in~\cite[Example~5.3]{KellerReiten07}.

The purpose of this note is to extend this class of examples by showing that
given $m \geq 3$, \emph{any} finite-dimensional algebra over an algebraically
closed field is the endomorphism algebra of an $m$-cluster-tilting object in an
$m$-Calabi-Yau triangulated category.
Hence, without any further assumptions on $\cC$ and $T$, one cannot say too
much about the algebras $\End_{\cC}(T)$.

To this end we invoke the construction of generalized cluster categories
due to Amiot~\cite{Amiot09} in the case $m=2$ and generalized by
Guo~\cite{Guo11} to the case $m>2$. This construction produces an $m$-Calabi-Yau
triangulated category with an $m$-cluster-tilting object from any dg-algebra
which is homologically smooth, bimodule $(m+1)$-Calabi-Yau and satisfies
additional finiteness conditions. A rich source of such dg-algebras is provided
by the deformed Calabi-Yau completions defined and investigated by
Keller~\cite{Keller11}.

Given a basic finite-dimensional algebra $A$, we choose a dg-algebra $B$ 
with two properties; firstly, the underlying graded algebra of $B$ is the path
algebra of a graded quiver whose arrows are concentrated in degrees $0$ and
$-1$, and secondly, $\hh^0(B) \cong A$. Such dg-algebra can be constructed
from any presentation of $A$ as a quotient of a path algebra of a quiver by an 
ideal generated by a finite sequence of elements. Conversely, any such
dg-algebra arises in this way.

It turns out that for any $m \geq 2$, the $(m+1)$-Calabi-Yau completion of $B$
is a Ginzburg dg-algebra $\Gamma$ of a graded quiver with homogeneous 
superpotential of degree $2-m$ which can be written explicitly in terms of the
quiver and the sequence of elements.
 In the case $m=2$, the zeroth homology
$\hh^0(\Gamma)$ is a split extension of $A$, whereas when $m>2$ it is
isomorphic to $A$, hence $\Gamma$ satisfies the finiteness conditions required
in the construction of~\cite{Amiot09,Guo11} and thus gives rise to a
$\Hom$-finite $m$-Calabi-Yau triangulated category with an $m$-cluster-tilting
object whose endomorphism algebra is isomorphic to~$A$.

The $m$-cluster-tilting object we get almost never satisfies the vosnex 
condition~\eqref{e:vosnex}. More precisely, that condition holds if and only if
$A$ is the path algebra of an acyclic quiver and $B$ is chosen such that $B=A$.
Moreover, the flexibility in the choice of the dg-algebra $B$ allows to
construct, for certain algebras $A$ and any $m>2$, inequivalent triangulated
$m$-Calabi-Yau categories $\cC \not \simeq \cC'$ with $m$-cluster-tilting
objects $T \in \cC$ and $T' \in \cC'$ such that
$\End_{\cC}(T) \cong \End_{\cC}(T') \cong A$. 

\section{Recollections}

\subsection{Notations}
\label{ssec:notat}

We recall the definitions of Calabi-Yau triangulated categories and
cluster-tilting objects. Throughout, we fix a field $K$.

\begin{defn}
Let $\cC$ be a $K$-linear triangulated category with suspension functor
$\Sigma$.
\begin{enumerate}
\renewcommand{\theenumi}{\alph{enumi}}
\item
We say that $\cC$ is \emph{$\Hom$-finite} if the spaces $\Hom_{\cC}(X,Y)$ are
finite-dimensional over $K$ for any $X, Y \in \cC$.

\item
Let $m \in \bZ$.
We say that $\cC$ is \emph{$m$-Calabi-Yau} if $\cC$ is $\Hom$-finite and
there exist functorial isomorphisms
\[
\Hom_{\cC}(X,Y) \cong D\Hom_{\cC}(Y, \Sigma^m X)
\]
for any $X, Y \in \cC$,
where $D$ denotes the duality $D(-) = \Hom_K(-, K)$.
\end{enumerate}
\end{defn}

For an object $X$ in an additive category $\cC$, denote by
$\add X$ the full subcategory of $\cC$
whose objects are finite direct sums of direct summands of $X$.
For a subcategory $\cY$ of $\cC$, let
${^\perp}\cY = \left\{ X \in \cC \,:\, \Hom_{\cC}(X,Y)=0
\text{ for any $Y \in \cY$} \right\}$.

\begin{defn}
Let $\cC$ be a triangulated $m$-Calabi-Yau category for some integer
$m \geq 1$.
An object $T$ of $\cC$ is \emph{$m$-cluster-tilting} if:
\begin{enumerate}
\renewcommand{\theenumi}{\roman{enumi}}
\item
$\Hom_{\cC}(T, \Sigma^i T) = 0$ for any $0 < i < m$; and

\item
if $X \in \cC$ is such that $\Hom_{\cC}(X, \Sigma^i T) = 0$ 
for any $0 < i < m$, then $X \in \add T$.
\end{enumerate}
Equivalently, $\add T = \bigcap_{0 < i < m} {^\perp} \Sigma^i(\add T)$.
\end{defn}

\begin{defn}
Let $m \geq 1$.
A $K$-algebra $\Lambda$ is \emph{$m$-Calabi-Yau-tilted}
(\emph{$m$-CY-tilted} for short)
if there exist a triangulated $m$-Calabi-Yau category $\cC$ and an
$m$-cluster-tilting object $T$ of $\cC$ such that $\Lambda \cong \End_{\cC}(T)$.
\end{defn}

\subsection{Generalized cluster categories}

\sloppy
In this section we briefly review the construction of 
triangulated $m$-Calabi-Yau categories with $m$-cluster-tilting object
due to Amiot~\cite{Amiot09} (for the case $m=2$)
and its generalization by Guo~\cite{Guo11} (for the case $m>2$).

Let $\Gamma$ be a differential graded (dg) algebra over a field $K$.
Denote by $\cD(\Gamma)$ its derived category and by $\per \Gamma$ its 
smallest full triangulated subcategory containing $\Gamma$ and closed under
taking direct summands.
Let $\cD_{\fd}(\Gamma)$ denote the full subcategory of $\cD(\Gamma)$
whose objects are those of $\cD(\Gamma)$ with finite-dimensional total
homology.

\begin{defn}
$\Gamma$ is said to be \emph{homologically smooth} if $\Gamma \in \per \Gamma^e$, where $\Gamma^e = \Gamma^{op} \otimes_K \Gamma$.
\end{defn}

Let $\Gamma$ be homologically smooth and let
$\Omega = \RHom_{\Gamma^e}(\Gamma, \Gamma^e)$.
By definition, $\Omega \in \cD((\Gamma^e)^{op})$. We can view $\Omega$
as an object of $\cD(\Gamma^e)$ via restriction of scalars along the 
morphism $\tau \colon \Gamma^e \xrightarrow{\sim} (\Gamma^e)^{op}$ given by
$\tau(x \otimes y) = y \otimes x$.
By~\cite[Lemma~3.4]{Keller11}
(see also~\cite[Lemma~4.1]{Keller08} and~\cite[Lemma~2.1]{Guo11})
one has $\cD_{\fd}(\Gamma) \subseteq \per \Gamma$ and
\begin{equation} \label{e:Serre}
\Hom_{\cD(\Gamma)}(L \otimes_{\Gamma} \Omega, M) \cong
D \Hom_{\cD(\Gamma)}(M, L)
\end{equation}
for any $L \in \cD(\Gamma)$, $M \in \cD_{\fd}(\Gamma)$.

\begin{defn}
Let $m \in \bZ$. If $\RHom_{\Gamma^e}(\Gamma, \Gamma^e) \cong
\Sigma^{-m} \Gamma$ in $\cD(\Gamma^e)$
we say that $\Gamma$ is \emph{bimodule $m$-Calabi-Yau}.
\end{defn}

If $\Gamma$ is homologically smooth and
bimodule $m$-Calabi-Yau then the triangulated category
$\cD_{\fd}(\Gamma)$ is $m$-Calabi-Yau. More precisely,
\eqref{e:Serre} yields functorial isomorphisms
\[
\Hom_{\cD(\Gamma)}(\Sigma^{-m} L, M) \cong
D \Hom_{\cD(\Gamma)}(M, L)
\]
for any $L \in \cD(\Gamma)$, $M \in \cD_{\fd}(\Gamma)$.

\begin{theorem}[\protect{\cite[\S2]{Amiot09}},
\protect{\cite[\S2]{Guo11}}]
\label{t:cluster}
Let $m \geq 1$ and
let $\Gamma$ be a dg-algebra satisfying the following conditions:
\begin{enumerate}
\renewcommand{\theenumi}{\roman{enumi}}
\item
$\Gamma$ is homologically smooth;

\item
$\hh^i(\Gamma)=0$ for any $i>0$;

\item
$\dim_K \hh^0(\Gamma) < \infty$;

\item
$\Gamma$ is bimodule $(m+1)$-Calabi-Yau.
\end{enumerate}

Consider the triangulated category
$\cC_{\Gamma} = \per \Gamma / \cD_{\fd}(\Gamma)$.
Then:
\begin{enumerate}
\renewcommand{\theenumi}{\alph{enumi}}
\item
$\cC_{\Gamma}$ is $\Hom$-finite and $m$-Calabi-Yau.

\item
For any $i \in \bZ$, set
$\cD^{\leq i} = \left\{ X \in \cD(\Gamma) \,:\, \hh^p(X)=0
\text{ for all $p>i$} \right\}$ and let
$\cF = \cD^{\leq 0} \cap {^\perp}\cD^{\leq -m} \cap \per \Gamma$.
The restriction of the canonical projection $\pi \colon \per \Gamma
\to \cC_{\Gamma}$
to $\cF$ induces an equivalence of $K$-linear categories
$\cF \xrightarrow{\sim} \cC_{\Gamma}$.

\item \label{it:t:CT}
The image $\pi \Gamma$ of $\Gamma$ in  $\cC_{\Gamma}$
is an $m$-cluster-tilting object and
$\End_{\cC_{\Gamma}}(\pi \Gamma) \cong \hh^0(\Gamma)$.
\end{enumerate}
\end{theorem}

The object $\pi \Gamma$ occurring in part~\eqref{it:t:CT} of the theorem 
is called the \emph{canonical $m$-cluster-tilting} object in $\cC_{\Gamma}$.
The statement of the next lemma concerning negative extension groups of the 
canonical $m$-cluster-tilting object is implicit in the proof
of~\cite[Corollary~3.4]{Guo11}.

\begin{lemma} \label{l:snex}
Let $\Gamma$ be a dg-algebra satisfying the conditions of
Theorem~\ref{t:cluster} and
let $T=\pi \Gamma$ be the canonical $m$-cluster-tilting object in
$\cC_{\Gamma}$. Then $\Hom_{\cC}(T, \Sigma^{-i} T) \cong \hh^{-i}(\Gamma)$
for any $0 \leq i \leq m-1$.
\end{lemma}
\begin{proof}
As observed in~\cite{Guo11}, the objects $\Gamma, \Sigma \Gamma, \dots,
\Sigma^{m-1} \Gamma$ are in the fundamental domain $\cF$.
Therefore, for any $0 \leq i \leq m-1$, 
\begin{align*}
\Hom_{\cC_{\Gamma}}(\pi \Gamma, \Sigma^{-i} \pi \Gamma) &\cong
\Hom_{\cC_{\Gamma}}(\Sigma^i \pi \Gamma, \pi \Gamma) 
= \Hom_{\cC_{\Gamma}}(\pi \Sigma^i \Gamma, \pi \Gamma) \cong
\Hom_{\cD(\Gamma)}(\Sigma^i \Gamma, \Gamma) \\
&\cong
\Hom_{\cD(\Gamma)}(\Gamma, \Sigma^{-i} \Gamma) \cong \hh^{-i}(\Gamma) .
\end{align*}
\end{proof}

\subsection{Ginzburg dg-algebras}
Ginzburg dg-algebras were introduced by Ginzburg in~\cite{Ginzburg06}.
We recall their definition in the case of a graded quiver with homogeneous
superpotential and quote the result of Keller~\cite{Keller11} that they
are homologically smooth and bimodule Calabi-Yau
(see also the paper~\cite{VandenBergh15} by Van den Bergh).
We note that a graded version of quivers with potentials (and their mutations)
has also been introduced in~\cite{AmiotOppermann14b} and~\cite{dTdVVdB13},
however in these papers one still implicitly considers Ginzburg dg-algebras
which are bimodule 3-Calabi-Yau.
In contrast, in the setting described below the degree of the superpotential
affects the Calabi-Yau dimension of its Ginzburg dg-algebra.

A \emph{quiver} is a finite directed graph. More precisely, it is a quadruple
$Q=(Q_0, Q_1, s, t)$, where $Q_0$ and $Q_1$ are finite sets
(of \emph{vertices} and \emph{arrows}, respectively) and
$s,t \colon Q_1 \to Q_0$ are functions specifying for each arrow its starting
and terminating vertex, respectively.
A quiver $Q$ is \emph{graded} if we are given a grading
$|\cdot| \colon Q_1 \to \bZ$.

A \emph{path} $p$ in $Q$ is a sequence of arrows
$\alpha_1 \alpha_2 \dots \alpha_n$
such that $s(\alpha_{i+1})=t(\alpha_i)$ for all $1 \leq i < n$.
For a path $p$ we denote by $s(p)$ its starting vertex $s(\alpha_1)$ and by
$t(p)$ its terminating vertex $t(\alpha_n)$.
A path $p$ is a \emph{cycle} if it starts and ends at the same vertex,
i.e.\ $s(p)=t(p)$. Any vertex $i \in Q_0$ gives rise to a cycle $e_i$ of
length zero with $s(e_i)=t(e_i)=i$.

The \emph{path algebra} $KQ$ has a basis consisting of the paths of $Q$,
and the product of two paths $p$ and $q$ is their concatenation $pq$
if $s(q)=t(p)$ and zero
otherwise. The path algebra is graded if $Q$ is graded, with the
degree of a path being the sum of the degrees of its arrows. The
degree of a homogeneous element $x$ of $KQ$ will be denoted by $|x|$.

Consider the $K$-bilinear map $[-,-] \colon KQ \times KQ \to KQ$
whose value on a pair of homogeneous elements $x, y$ is given by their
\emph{supercommutator} $[x,y] = xy - (-1)^{|x||y|}yx$. Denote by
$[KQ,KQ]$ the linear subspace of $KQ$ spanned by all the supercommutators.
The quotient $KQ/[KQ,KQ]$ has a basis consisting of cycles considered up to
cyclic permutation ``with signs''.

\begin{defn}
A \emph{superpotential} on $Q$ is a homogeneous element in $KQ/[KQ,KQ]$.
\end{defn}

Any arrow $\alpha \in Q_1$ gives rise to a linear map
$\partial_\alpha \colon KQ/[KQ,KQ] \to KQ$ 
(called \emph{cyclic derivative with respect to $\alpha$})
whose value on any cycle $p$ is given by
\begin{equation} \label{e:deriv}
\partial_\alpha p = (-1)^{|\alpha|}
\sum_{p=u \alpha v} (-1)^{|u|(|\alpha|+|v|)} vu
\end{equation}
where the sum runs over all possible decompositions $p=u \alpha v$ with
$u, v$ paths of length $\geq 0$.
More explicitly, if $p=\alpha_1 \alpha_2 \dots \alpha_n$ has degree $w$, then
\[
\partial_\alpha (\alpha_1 \alpha_2 \dots \alpha_n) = 
(-1)^{|\alpha|} \sum_{\ell \,:\, \alpha_{\ell} = \alpha}
(-1)^{(w-1)(|\alpha_1|+\dots+|\alpha_{\ell-1}|)}
\alpha_{\ell+1} \dots \alpha_n \alpha_1 \dots \alpha_{\ell-1}
\]
is homogeneous of degree $w-|\alpha|$.

\begin{defn} \label{def:Ginzburg}
Let $Q$ be a graded quiver and let $m \in \bZ$. Let $\wb{Q}$ be the graded
quiver whose vertices are those of $Q$ and its set of arrows consists of
\begin{itemize}
\item
the arrows of $Q$ (with their degree unchanged);
\item
an arrow $\alpha^* \colon j \to i$ of degree $1-m-|\alpha|$
for each arrow $\alpha \colon i \to j$ of $Q$;

\item
a loop $t_i \colon i \to i$ of degree $-m$ for each vertex $i \in Q_0$.
\end{itemize}

Let $W$ be a superpotential on $Q$ of degree $2-m$.
The \emph{Ginzburg dg-algebra of $(Q,W)$}, denoted $\Gamma_{m+1}(Q,W)$, is the
dg-algebra whose underlying graded algebra is the path algebra $K\wb{Q}$
and the differential $d$ is defined by its action on the generators as
\begin{itemize}
\item
$d(\alpha) = 0$ and $d(\alpha^*) = \partial_{\alpha} W$
for each $\alpha \in Q_1$;

\item
$d(t_i) = e_i (\sum_{\alpha \in Q_1} [\alpha, \alpha^*]) e_i$
for each $i \in Q_0$.
\end{itemize}
\end{defn}

\begin{remark}
Note that for each $\alpha \in Q_1$ the element $\partial_\alpha W$ is
homogeneous of degree $2-m-|\alpha|$ and the supercommutator
$[\alpha,\alpha^*]$ is homogeneous of degree $1-m$, hence the definition of the
differential $d$ makes sense. Moreover, as $[\alpha,\alpha^*]=\alpha \alpha^* -
(-1)^{|\alpha||\alpha^*|} \alpha^* \alpha$ for any $\alpha \in Q_1$, by using
the sign conventions in~\eqref{e:deriv} one verifies that $d^2(t_i)=0$ for any
$i \in Q_0$, so that $d$ is indeed a differential.
Note also that in~\cite{Keller11} the differential of $t_i$ is
$(-1)^{m-1}$ times the one given here, but of course by replacing each
$t_i$ by $(-1)^{m-1}t_i$ one sees that the dg-algebras are isomorphic.
\end{remark}

\begin{theorem}[\protect{\cite[Theorem~6.3]{Keller11}}]
\label{t:GinzburgCY}
Let $m \in \bZ$ and
let $W$ be a superpotential of degree $2-m$ on a graded quiver $Q$. Then
$\Gamma_{m+1}(Q,W)$ is homologically smooth and bimodule $(m+1)$-Calabi-Yau.
\end{theorem}

We record two useful observations. Let $Q$ be a graded quiver, $m \in \bZ$
an integer and $W$ a homogeneous superpotential on $Q$ of degree $2-m$.

\begin{lemma} \label{l:reparrow}
Suppose that $\alpha \colon i \to j$ is an arrow in $Q$ such that no term of
$W$ contains $\alpha$. Define a graded quiver $Q'$ by $Q'_0 = Q_0$ and
$Q'_1 = Q_1 \setminus \{\alpha\} \cup \{\alpha^*\}$ where $\alpha^* \colon
j \to i$ has degree $1-m-|\alpha|$. Then $W$ can be naturally viewed as
a superpotential $W'$ on $Q'$ and $\Gamma_{m+1}(Q,W) \cong \Gamma_{m+1}(Q',W')$.
\end{lemma}
\begin{proof}
Since no term of $W$ contains $\alpha$, we can view $W$ as an element $W'$ in
the path algebra $KQ'$. The graded quivers $\wb{Q'}$ and $\wb{Q}$ are
isomorphic by the map $\varphi \colon \wb{Q'} \to \wb{Q}$ sending $(\alpha^*)^*$
to $\alpha$ and fixing all other arrows.
Moreover, since $\partial_\alpha W = 0 = \partial_{\alpha^*} W'$ and
$\partial_\beta W = \partial_\beta W'$ for any $\beta \in Q_1 \cap Q'_1$,
the map $\varphi$ induces an isomorphism of the Ginzburg dg-algebras
$\Gamma_{m+1}(Q',W') \cong \Gamma_{m+1}(Q,W)$.
\end{proof}

For a subset of arrows $\Omega \subseteq Q_1$, let
$Q_{\Omega}$ be the subquiver of $Q$ with $(Q_{\Omega})_0=Q_0$ and
$(Q_{\Omega})_1 = \Omega$. For the definition of Calabi-Yau completion,
see~\cite[\S4]{Keller11}.

\begin{lemma} \label{l:CYcompl}
Suppose that $\Omega \subseteq Q_1$ is a set of arrows such that
$W = \sum_{\beta \in Q_1 \setminus \Omega} \beta \omega_\beta$ with
$\omega_\beta \in KQ_{\Omega}$ for each $\beta \not \in \Omega$.
Consider the subquiver $Q'$ of $\wb{Q}$ defined by $Q'_0 = Q_0$ and
$Q'_1 = \Omega \cup \{\beta^* : \beta \in Q_1 \setminus \Omega\}$.
Let $d$ be the differential on $\Gamma_{m+1}(Q,W)$. Then:
\begin{enumerate}
\renewcommand{\theenumi}{\alph{enumi}}
\item
$d(KQ') \subseteq KQ'$, hence $B=(KQ',d)$ is a sub-dg-algebra of
$\Gamma_{m+1}(Q,W)$.

\item
$\Gamma_{m+1}(Q,W)$ is isomorphic to the $(m+1)$-Calabi-Yau completion
of $B$.
\end{enumerate}
\end{lemma}
\begin{proof}
The first claim holds since $d(\alpha)=0$ for any
$\alpha \in \Omega \subseteq Q_1$ and
$d(\beta^*) = \partial_\beta W = (-1)^{|\beta|} \omega_\beta \in K Q_{\Omega}
\subseteq KQ'$ for any $\beta \not \in \Omega$. This argument also shows
that the differential $d$ on $KQ'$ satisfies the condition
in~\cite[\S3.6]{Keller11} via the filtration
$\varnothing \subseteq \Omega \subseteq Q'$ of the set of arrows.
Hence we can use~\cite[Proposition~6.6]{Keller11} to compute the
$(m+1)$-Calabi-Yau completion of $B$ and get that it is isomorphic to
$(K\wb{Q'},d')$ with the differential $d'$ given on the generators by
\begin{align*}
d'(\alpha) = \partial_{\alpha^*} W' = 0 
&,& d'((\beta^*)^*) = \partial_{\beta^*} W' = 0
&,& d'(\alpha^*) = \partial_\alpha W'
&,& d'(\beta^*) = \partial_{(\beta^*)^*} W'
\end{align*}
and $d(t_i) = (-1)^{m+1} e_i (\sum_{\gamma \in Q'_1} [\gamma, \gamma^*]) e_i$,
where $W' \in K\wb{Q'}$ is the element
\[
W' = \sum_{\alpha \in \Omega} (-1)^{|\alpha|} \alpha^* d(\alpha)
+ \sum_{\beta \in Q_1 \setminus \Omega} (-1)^{|\beta^*|} (\beta^*)^* d(\beta^*)
= \sum_{\beta \in Q_1 \setminus \Omega} (-1)^{m-1} (\beta^*)^* \omega_\beta .
\]

Finally, the isomorphism $\varphi \colon K\wb{Q'} \to K\wb{Q}$ defined on the
generators by
\[
\varphi(\gamma) = \begin{cases}
(-1)^{m-1} \gamma & \text{if $\gamma=\alpha^*$ for some $\alpha \in Q_1$}, \\
\beta & \text{if $\gamma=(\beta^*)^*$ for some
$\beta \in Q_1 \setminus \Omega$}, \\
\gamma & \text{otherwise}
\end{cases}
\]
induces an isomorphism $(K\wb{Q'},d') \cong \Gamma_{m+1}(Q,W)$ of dg-algebras.
\end{proof}

\subsection{Non-positively graded Ginzburg dg-algebras}
In this section we restrict attention to Ginzburg dg-algebras which are
concentrated in non-positive degrees and quote the construction of
Guo~\cite{Guo11} (generalizing that of Amiot in~\cite{Amiot09} for the case
$m=2$) of the generalized cluster category associated to a quiver with
superpotential.

Let $Q$ be a graded quiver and let $w \in \bZ$.
Denote by $Q^{(w)}$ the subquiver of $Q$ consisting of the arrows of
degree $w$, in other words, $Q^{(w)}_0 = Q_0$ and
$Q^{(w)}_1 = \{ \alpha \in Q_1 \,:\, |\alpha|=w \}$.

Now fix a graded quiver $Q$ and let $m \in \bZ$.
Let $\wb{Q}$ denote the quiver constructed in Definition~\ref{def:Ginzburg}.
Fix a homogeneous superpotential $W$ on $Q$ of degree $2-m$ and let
$\Gamma=\Gamma_{m+1}(Q,W)$ be the Ginzburg dg-algebra.
We immediately observe:
\begin{remark}
$\Gamma^i=0$ for any $i>0$ if and only if all the arrows of $\wb{Q}$ have
non-positive degrees. This condition implies that $\hh^i(\Gamma)=0$ for any
$i>0$.
\end{remark}

\begin{lemma} \label{l:h0}
Assume that all the arrows of $\wb{Q}$ have non-positive degrees. Then
$\hh^0(\Gamma) \cong K{\wb{Q}}^{(0)}/(d\alpha : \alpha \in \wb{Q}^{(-1)}_1)$.
\end{lemma}
\begin{proof}
Since there are no arrows of $\wb{Q}$ of positive degree, the graded piece
$\Gamma^0$ equals the path algebra of $\wb{Q}^{(0)}$ and the graded piece
$\Gamma^{-1}$ is spanned by the elements of the form $u \alpha v$ where
$u,v \in K{\wb{Q}}^{(0)}$ and $\alpha \in \wb{Q}^{(-1)}_1$. As $d(u \alpha v) = u(d\alpha)v$, the claim follows.
\end{proof}

\begin{lemma} \label{l:wbQnonpos}
The following conditions are equivalent:
\begin{enumerate}
\renewcommand{\theenumi}{\alph{enumi}}
\item
Each arrow of $\wb{Q}$ has non-positive degree;

\item
$m \geq 0$ and for any $\alpha \in Q_1$ one has
$1-m \leq |\alpha|$ and $|\alpha| \leq 0$.
\end{enumerate}
\end{lemma}
\begin{proof}

If $m<0$ then the each of the loops $t_i$ in $\wb{Q}$ has positive degree $-m$.
In addition, each other arrow of $\wb{Q}$ is either $\alpha$ or $\alpha^*$ for
some arrow $\alpha \in Q_1$, and its degree is $|\alpha|$ or $1-m-|\alpha|$,
respectively. These observations imply the statement of the lemma.
\end{proof}

In the next two examples we discuss the cases $m=0$ and $m=1$.

\begin{example}
Assume that $m=0$. Lemma~\ref{l:wbQnonpos} implies that the arrows of $\wb{Q}$
have non-positive degrees if and only if $Q$ has no arrows. In this case
$\wb{Q}$ is a disjoint union of graded quivers of the form
\[
\xymatrix{
{\bullet} \ar@{-}@(ur,dr)[]^{t}
}
\]
with $|t|=0$ and $dt=0$. Hence $\Gamma$ is concentrated in degree $0$ and it
is a finite direct product of polynomial rings $k[t]$. The algebra $k[t]$ is
$1$-Calabi-Yau, see~\cite[\S4.2]{Keller08}
\end{example}

\begin{example}
Assume that $m=1$. Lemma~\ref{l:wbQnonpos} implies that the arrows of $\wb{Q}$
have non-positive degrees if and only if all the arrows of $Q$ are in degree
$0$, so we can regard $Q$ as an ungraded quiver. Since the superpotential on
$Q$ is homogeneous of degree $1$, it must vanish.
In this case all the arrows $\alpha$ and $\alpha^*$ of $\wb{Q}$ are
in degree $0$ and the differential of each loop $t_i$, whose degree is $-1$,
is given by $d(t_i) = e_i (\sum_{\alpha \in Q_1} [\alpha,\alpha^*]) e_i$.
By Lemma~\ref{l:h0}, $\hh^0(\Gamma)$ is isomorphic to the
\emph{preprojective algebra} of the quiver $Q$.
\end{example}

When $m \geq 2$, the next lemma shows that
we may assume that $Q^{(1-m)}$ has no arrows.

\begin{lemma}
Assume that $m \geq 2$ and that $1-m \leq |\alpha| \leq 0$ for any
$\alpha \in Q_1$. Then there exist a graded quiver $Q'$ with
$2-m \leq |\alpha'| \leq 0$ for each $\alpha' \in Q'_1$
and a homogeneous superpotential $W'$ on $Q'$ of degree $2-m$ such that
$\Gamma_{m+1}(Q,W) \cong \Gamma_{m+1}(Q',W')$.
\end{lemma}
\begin{proof}
We define the graded quiver $Q'$ as a subquiver of $\wb{Q}$; 
we set $Q'_0=Q_0$ and
\[
Q'^{(w)}_1 = \begin{cases}
Q_1^{(0)} \cup \{\alpha^* : \alpha \in Q_1^{(1-m)}\} &
\text{if $w=0$,} \\
Q_1^{(w)} & \text{if $1-m < w < 0$,}
\end{cases}
\]
with $Q'^{(w)}_1$ empty for any other $w \in \bZ$. Then
$2-m \leq |\alpha'| \leq 0$ for any arrow $\alpha' \in Q'_1$ by construction.
Moreover, since the superpotential $W$ is of degree $2-m$ and all the arrows of
$Q$ have non-positive degrees, no term of $W$ can contain any arrows of
$Q^{(1-m)}$. The result now follows by iterated application of
Lemma~\ref{l:reparrow} for each of the arrows in $Q^{(1-m)}$.
\end{proof}

In particular, when $m=2$ one can always reduce to the classical setting of
an ungraded quiver with potential.

\begin{cor}
Let $Q$ be a graded quiver such that $-1 \leq |\alpha| \leq 0$ for any
$\alpha \in Q_1$ and let $W$ be a homogeneous superpotential on $Q$ of degree
$0$. Then there exist a quiver $Q'$ concentrated in degree $0$ and a
superpotential $W'$ on $Q'$ such that $\Gamma_3(Q,W) \cong \Gamma_3(Q',W')$.
\end{cor}

The next lemma generalizes ~\cite[Lemma~2.11]{KellerYang11}.

\begin{lemma} \label{l:h0Q}
Assume that $m \geq 2$ and $2-m \leq |\alpha| \leq 0$ for any $\alpha \in Q_1$.
Then
\[
\hh^0(\Gamma_{m+1}(Q,W)) \cong
KQ^{(0)}/(\partial_\alpha W : \alpha \in Q^{(2-m)}_1) .
\]
\end{lemma}
\begin{proof}
Observe that since the arrows of $Q$ have non-positive degrees and $W$ is
of degree $2-m$, the cyclic derivative $\partial_\alpha W$ with respect
to any arrow $\alpha$ of degree $2-m$ lies in the path algebra of
$Q^{(0)}$ and the quotient in the right hand side makes sense.

The arrows of $\wb{Q}$ have non-positive degrees by Lemma~\ref{l:wbQnonpos},
hence we may apply Lemma~\ref{l:h0} to deduce that
$\hh^0(\Gamma_{m+1}(Q,W)) \cong K \wb{Q}^{(0)}/(d \alpha : 
\alpha \in \wb{Q}_1^{(-1)})$. 
Observe that $\wb{Q}^{(0)} = Q^{(0)}$ since $Q_1^{(1-m)}$ is empty by
assumption and
$\wb{Q}^{(-1)}_1 = Q^{(-1)}_1 \cup \{\alpha^* : \alpha \in Q^{(2-m)}_1 \}$
with $d(\alpha)=0$ and $d(\alpha^*) = \partial_\alpha W$ for any
$\alpha \in Q_1$, hence the claim follows.
\end{proof}

\begin{theorem}[\protect{\cite[Theorem~3.3]{Guo11}}] \label{t:mCY}
Let $m \geq 1$ be an integer, let $Q$ be a graded quiver,
let $W$ be a superpotential on $Q$ of degree $2-m$ and denote by
$\Gamma = \Gamma_{m+1}(Q,W)$ the Ginzburg dg-algebra.
Assume that:
\begin{enumerate}
\renewcommand{\theenumi}{\roman{enumi}}
\item \label{it:t:deg}
$1 - m \leq |\alpha| \leq 0$ for all $\alpha \in Q_1$;
\item \label{it:t:fd}
$\hh^0(\Gamma)$ is finite-dimensional.
\end{enumerate}

Then the triangulated category
$\cC_{(Q,W)} = \per \Gamma / \cD_{\fd}(\Gamma)$ is
$\Hom$-finite and $m$-Calabi-Yau. Moreover, the image of $\Gamma$
in $\cC_{(Q,W)}$ is an $m$-cluster-tilting object whose endomorphism algebra
is isomorphic to $\hh^0(\Gamma)$.
\end{theorem}
\begin{proof}
The condition~\eqref{it:t:deg} implies that $\hh^i(\Gamma)=0$ for all
$i>0$. Now Theorem~\ref{t:GinzburgCY} and condition~\eqref{it:t:fd}
imply that $\Gamma$ satisfies the conditions of Theorem~\ref{t:cluster}
and the result is now a consequence of that theorem.
\end{proof}

\section{The construction}

\subsection{Superpotentials from quivers with relations}

In this section we construct, given a quiver $Q$, a finite sequence $R$ of
relations on $Q$ and an integer $m \geq 2$, a graded quiver with homogeneous
superpotential of degree $2-m$. The construction generalizes that of
Keller in~\cite[\S6.9]{Keller11} for the case where $m=2$ and $KQ/(R)$ has
global dimension $2$. 
The idea of adding, for each relation, an arrow in the opposite direction
appears already in the description of relation-extension algebras by
Assem, Br\"{u}stle and Schiffler~\cite{ABS08}.

\begin{defn}
Let $Q$ be a quiver and let $\fr$ be the ideal of $KQ$ generated by all
the arrows.
A \emph{relation} on $Q$ is an element of $e_i \fr e_j$ for some $i, j \in Q_0$.
In other words, a relation is a linear combination of paths of positive lengths
starting at $i$ and ending at~$j$.
\end{defn}

We start with some preparations concerning split extensions of algebras which
will be needed for the case $m=2$.

\begin{defn}
An algebra $\wt{A}$ is a \emph{split extension} of an algebra $A$
if there exist algebra homomorphisms $\iota \colon A \to \wt{A}$ and
$\pi \colon \wt{A} \to A$ such that $\pi \iota = id_A$.
\end{defn}

Let $Q$ be a quiver and let $\wt{Q}$ be a quiver such that
$\wt{Q}_0=Q_0$ and $Q_1 \subseteq \wt{Q}_1$ (in other words, $\wt{Q}$ is
obtained from $Q$ by adding arrows). Then the path algebra $K\wt{Q}$ is a split
extension of the path algebra $KQ$. Indeed, there are algebra homomorphisms
\[
KQ \xrightarrow{\iota_{Q,\wt{Q}}} K\wt{Q} \xrightarrow{\pi_{\wt{Q},Q}} KQ
\]
whose values on the generators are given by
\begin{align*}
\iota_{Q,\wt{Q}}(\alpha) = \alpha \quad (\alpha \in Q_1)
\qquad \text{and} \qquad
\pi_{\wt{Q},Q}(\alpha) = \begin{cases}
\alpha & \text{if $\alpha \in Q_1$,} \\
0 & \text{if $\alpha \in \wt{Q}_1 \setminus Q_1$.}
\end{cases}
\end{align*}

Denote by $\fr'$ the ideal of $K\wt{Q}$ generated by the arrows in the set
$\wt{Q}_1 \setminus Q_1$.
Let $R$ be a set of relations in $KQ$ and let $\wt{R}$ be a set of
relations in $K\wt{Q}$ such that $R \subseteq \wt{R}$ (via the natural
embedding $\iota_{Q,\wt{Q}} \colon KQ \hookrightarrow K\wt{Q}$).

\begin{lemma} \label{l:splitext}
Assume that $\wt{R} \setminus R \subseteq \fr'$.
Then the algebra $K\wt{Q}/(\wt{R})$ is a split extension of 
the algebra $KQ/(R)$.
\end{lemma}
\begin{proof}
Consider the composition
$KQ \xrightarrow{\iota_{Q,\wt{Q}}} K\wt{Q} \twoheadrightarrow K\wt{Q}/(\wt{R})$.
Since $R \subseteq \wt{R}$, the image of any relation $\rho \in R$ vanishes and
we get an algebra homomorphism $\iota \colon KQ/(R) \to K\wt{Q}/(\wt{R})$.
Consider now the composition
$K\wt{Q} \xrightarrow{\pi_{\wt{Q},Q}} KQ \twoheadrightarrow KQ/(R)$ and let
$\rho \in \wt{R}$. If $\rho \in R$, then its image obviously vanishes.
Otherwise, $\rho \not \in R$ and our assumption that $\wt{R} \setminus R
\subseteq \fr'$ implies that $\pi_{\wt{Q},Q}(\rho)=0$, so its image vanishes
as well. Hence we get a well defined algebra homomorphism
$\pi \colon K\wt{Q}/(\wt{R}) \to KQ/R$. The composition $\pi \iota$ maps
the image of any arrow $\alpha \in Q$ in $KQ/(R)$ to itself, therefore it
is the identity on $KQ/(R)$.
\end{proof}

Any quiver with a finite sequence of relations gives rise to a dg-algebra whose
underlying graded algebra is the path algebra of a graded quiver with arrows
concentrated in degrees $0$ and $-1$. The details are given in the construction
below.

\begin{constr} \label{con:dgQR}
Let $(Q,R)$ be a pair where $Q$ is a quiver and
$R = \bigcup_{i,j \in Q_0} R_{i,j}$, where each $R_{i,j}$ is a finite sequence
of relations inside $e_i \fr e_j$ and $R$ is the concatenation of these
sequences (there may be repetitions inside each sequence $R_{i,j}$ and moreover
the zero element can appear inside several such sequences).
For a relation $\rho \in R_{i,j}$, set $s(\rho)=i$ and $t(\rho)=j$.

We define a graded quiver $Q'$ as follows:
\begin{itemize}
\item
The set of vertices of $Q'$ equals that of $Q$;
\item
The set of arrows consists of
\begin{itemize}
\item
the arrows of $Q$, with their degree set to $0$;
\item
an arrow $\eta_\rho \colon s(\rho) \to t(\rho)$ of degree $-1$ for each
relation $\rho \in R$;
\end{itemize}
\end{itemize}
and denote by $B(Q,R)$ the dg-algebra whose underlying graded algebra is the
path algebra $KQ'$ with the differential acting on the generators by
\begin{itemize}
\item
$d(\alpha)=0$ for any $\alpha \in Q_1$;
\item
$d(\eta_\rho)=\rho$ for any $\rho \in R$.
\end{itemize}
\end{constr}

Obviously, $B(Q,R)$ is concentrated in non-positive degrees and $\hh^0(B(Q,R))$ is isomorphic to $KQ/(R)$, but in general $B(Q,R)$ is not quasi-isomorphic to its
zeroth homology.

\begin{remark}
If $B=(KQ',d)$ is any dg-algebra whose underlying graded algebra is the path
algebra of a graded quiver with arrows concentrated in degrees $0$ and
$-1$ and the image of the differential $d$ lies in the ideal generated by the
arrows, then $B \cong B(Q,R)$ for some $(Q,R)$. Indeed, we can take
$Q=Q'^{(0)}$ and for each $i,j \in Q_0$ let
$R_{i,j}$ be the list of $d(\alpha)$ where $\alpha$ runs over the arrows
in $Q'^{(-1)}_1$ starting at $i$ and ending at $j$.
\end{remark}

It turns out (see Lemma~\ref{l:preprojB} below) that for any $m \geq 2$, the
$(m+1)$-Calabi-Yau completion of $B(Q,R)$ is a Ginzburg dg-algebra of a graded
quiver with homogeneous superpotential of degree $2-m$ whose construction is 
described below.

\begin{constr} \label{con:QW}
Let $(Q,R,m)$ be a triple where $(Q,R)$ is as in Construction~\ref{con:dgQR}
and $m \geq 2$ is an integer.
We construct a graded quiver $\wt{Q}$ with homogeneous superpotential $W$ of
degree $2-m$ as follows.
\begin{itemize}
\item
The set of vertices of $\wt{Q}$ equals that of $Q$;

\item
The set of arrows of $\wt{Q}$ consists of
\begin{itemize}
\item
the arrows of $Q$, with their degree set to $0$;

\item
an arrow $\eps_\rho \colon t(\rho) \to s(\rho)$ of degree $2-m$
for each relation $\rho \in R$;
\end{itemize}

\item
The superpotential $W$ is the image of the element
$\sum_{\rho \in R} \eps_\rho \rho$ in $K\wt{Q}/[K\wt{Q},K\wt{Q}]$.
\end{itemize}

We denote by $\Gamma(Q,R,m)$ the Ginzburg dg-algebra $\Gamma_{m+1}(\wt{Q},W)$.
\end{constr}

\begin{remark}
The quiver with superpotential $(\wt{Q},W)$ depends on the particular choice
of the sequence $R$ and not only on the two-sided ideal $(R)$ it generates in
$KQ$, as we shall see in Example~\ref{ex:nonuniq}.
\end{remark}

For the rest of this section, we fix a triple $(Q,R,m)$ where $Q$ is a quiver,
$R$ is a finite sequence of relations on $Q$ and $m \geq 2$.
We denote by $(\wt{Q},W)$ the graded quiver with superpotential of degree
$2-m$ associated to $(Q,R,m)$ as in Construction~\ref{con:QW}, and by
$\Gamma = \Gamma(Q,R,m) = \Gamma_{m+1}(\wt{Q},W)$ its Ginzburg dg-algebra.
We denote the elements of $R$ by $\rho_1, \rho_2, \dots, \rho_n$ and
write $|R|=n$. For simplicity, we denote by $\eps_k$ the arrow $\eps_{\rho_k}$
of degree $2-m$ in $\wt{Q}$ corresponding to the relation $\rho_k$, so that
$W$ is the image of $\sum_{k=1}^n \eps_k \rho_k$ modulo $[K\wt{Q},K\wt{Q}]$.
Similarly, denote by $\eta_k$ the arrow $\eta_{\rho_k}$ and let
$B=B(Q,R)$ be the dg-algebra of Construction~\ref{con:dgQR}.

We start by describing the graded quiver $\wb{\wt{Q}}$ underlying $\Gamma$
occurring in Definition~\ref{def:Ginzburg}.
\begin{lemma} \label{l:QQarrows}
The arrows of the graded quiver $\wb{\wt{Q}}$, their degrees and their
differentials are as given in Table~\ref{tab:arrows}.
\end{lemma}
\begin{proof}
The description of the arrows and their degrees is evident from
Definition~\ref{def:Ginzburg}. For the differentials, note that
since none of the arrows $\eps_k$ occur in any $\rho \in R$,
we have 
\[
\partial_{\eps_k}W =
\partial_{\eps_k} \Bigl(\sum_{\ell=1}^{n} \eps_{\ell} \rho_{\ell} \Bigr) =
\partial_{\eps_k} (\eps_k \rho_k) = (-1)^{|\eps_k|} \rho_k =
(-1)^m \rho_k
\]
for any $1 \leq k \leq n$.
\end{proof}

\begin{table}
\[
\begin{array}{cccl}
\text{arrow} & \text{degree} & \text{differential} \\ \hline
\alpha   & 0   & 0 & (\alpha \in Q_1) \\
\eps_k^* & -1  & (-1)^m \rho_k & (1 \leq k \leq n) \\
\eps_k   & 2-m & 0 & (1 \leq k \leq n) \\
\alpha^* & 1-m & \partial_{\alpha} W & (\alpha \in Q_1) \\
t_i      & -m  & d(t_i) & (i \in Q_0)
\end{array}
\]
\caption{The arrows of the graded quiver $\wb{\wt{Q}}$ and their
differentials.}
\label{tab:arrows}
\end{table}

\begin{lemma} \label{l:H0}
Consider the algebras $A = KQ/(R)$ and $\wt{A} = \hh^0(\Gamma)$.
\begin{enumerate}
\renewcommand{\theenumi}{\alph{enumi}}
\item \label{it:l:mgt2}
If $m>2$ then $\wt{A} \cong A$.

\item \label{it:l:meq2}
If $m=2$ then $\wt{A}$ is a split extension of $A$.
\end{enumerate}
\end{lemma}
\begin{proof}
By Lemma~\ref{l:h0Q}, $\wt{A} \cong K\wt{Q}^{(0)}/(\partial_\alpha W :
\alpha \in \wt{Q}_1^{(2-m)})$.

If $m>2$ then $\wt{Q}^{(0)} = Q$, the arrows of $\wt{Q}^{(2-m)}$ are
$\eps_1, \eps_2, \dots, \eps_n$ and $\partial_{\eps_k} W = (-1)^m \rho_k$
according to Table~\ref{tab:arrows}.
Hence $\wt{A} \cong KQ/(\rho_1, \rho_2, \dots, \rho_n)$.
This shows part~\eqref{it:l:mgt2}.

If $m=2$ then all the arrows of the quiver $\wt{Q}$ have degree $0$, and we
can think of $\wt{Q}$ as an ungraded quiver consisting of the arrows of $Q$
and the arrows $\eps_k$ for $1 \leq k \leq n$. In other words, 
$\wt{Q} = \wt{Q}^{(0)}$ and $Q \subseteq \wt{Q}$ with
$\wt{Q} \setminus Q = \{\eps_1, \eps_2, \dots, \eps_n\}$.
Moreover, $\wt{A} = K\wt{Q}/(\wt{R})$ for $\wt{R} = \{\partial_\alpha W : 
\alpha \in \wt{Q}_1\}$.
Now $R \subseteq \wt{R}$ since $\partial_{\eps_k} W = \rho_k$ for each
$1 \leq k \leq n$. In addition, for any $\alpha \in Q_1$ the element
$\partial_\alpha W = \sum_{k=1}^n \partial_\alpha (\eps_k \rho_k)$ lies in the
ideal generated by $\eps_1, \eps_2, \dots, \eps_n$ in $\wt{Q}$.
Part~\eqref{it:l:meq2} now follows from Lemma~\ref{l:splitext}.
\end{proof}

The next two lemmas relate the dg-algebra $B$ and the Ginzburg dg-algebra
$\Gamma$. For a similar result in the case $m=2$,
see~\cite[\S6.7 and Proposition~6.8]{Keller11}.

\begin{lemma} \label{l:preprojB}
$\Gamma$ is the $(m+1)$-Calabi-Yau completion of $B$.
\end{lemma}
\begin{proof}
This is a consequence of Lemma~\ref{l:CYcompl}.
Indeed, we may take $\Omega = Q_1$ so that $\wt{Q}_1 \setminus \Omega =
\{\eps_1, \dots, \eps_n\}$ and the superpotential $W = \sum_{k=1}^n
\eps_k \rho_k$ has the required form. We only note that the differential
of $\Gamma$ restricts to the path algebra of the quiver with arrows
$Q_1 \cup \{\eps^*_1, \dots, \eps^*_n\}$ and the resulting dg-algebra is
isomorphic to $B$ by mapping each arrow $\eps^*_k$ to $(-1)^m \eta_{\rho_k}$
and sending each $\alpha \in Q_1$ to itself.
\end{proof}

\begin{lemma}
$\hh^{-i}(\Gamma) \cong \hh^{-i}(B)$ for any $i < m-2$.
\end{lemma}
\begin{proof}
From Table~\ref{tab:arrows} we see that $\Gamma^{-i} \cong B^{-i}$ for any
$i < m-2$, hence $\hh^{-i}(\Gamma) \cong \hh^{-i}(B)$ for any $i < m-3$.
The graded piece $\Gamma^{2-m}$ can be decomposed as
$\Gamma^{2-m} \cong B^{2-m} \oplus F$, where the space $F$ is spanned by
the elements of the form $u \eps_k v$ where $1 \leq k \leq n$
and $u, v$ are paths in $Q$. 
Since the differential of each such element vanishes, one has $d(F)=0$,
hence $d(\Gamma^{2-m}) \cong d(B^{2-m})$ and therefore $\hh^{3-m}(\Gamma)
\cong \hh^{3-m}(B)$ as well.
\end{proof}

\begin{lemma} \label{l:Qempty}
If $R$ is empty then $\hh^{-i}(\Gamma)=0$ for any $1 \leq i \leq m-2$.
\end{lemma}
\begin{proof}
If $R$ is empty, then by Lemma~\ref{l:QQarrows} the arrows in the graded
quiver $\wb{\wt{Q}}$ have degrees $0$, $1-m$ or $-m$. Hence the graded piece
$\Gamma^{-i}$ vanishes for each $1 \leq i \leq m-2$ and the claim follows.
\end{proof}

The next lemma provides a partial converse to Lemma~\ref{l:Qempty}.

\begin{lemma} \label{l:H2mR}
If $m>2$ and $R \subseteq \fr^2$, then $\dim_K \hh^{2-m}(\Gamma) \geq
|R|$.
\end{lemma}
\begin{proof}
The graded piece $\Gamma^{2-m}$ is spanned by two types of elements:
\begin{enumerate}
\item \label{it:typ01}
$u_1 \eps_{k_1}^* u_2 \eps_{k_2}^* u_3 \dots \eps_{k_{m-2}}^* u_{m-1}$,
where $u_1, u_2, \dots, u_{m-1}$ are paths in $Q$ and $1 \leq k_j \leq n$
for each $1 \leq j \leq m-2$;

\item \label{it:typ02}
$u \eps_k v$, where $u, v$ are paths in $Q$ and $1 \leq k \leq n$;
\end{enumerate}

As a $K$-vector space, we may thus decompose $\Gamma^{2-m}$ into a direct sum
$E \oplus E'$, where $E$ is the $n$-dimensional subspace
$E = \{\lambda_1 \eps_1 + \lambda_2 \eps_2 + \dots + \lambda_n \eps_n :
(\lambda_1, \lambda_2, \dots, \lambda_n) \in K^n\}$
and $E'$ is spanned by all the elements of type~\eqref{it:typ01} and those of
type~\eqref{it:typ02} such that at least one of $u, v$ has positive length.

We will show that $d(\Gamma^{1-m}) \subseteq E'$ and hence no non-zero element
in $E$ lies in the image of the differential $d$ acting on the graded piece
$\Gamma^{1-m}$. Since $d$ vanishes on $E$,
this will yield an $n$-dimensional subspace inside $\hh^{2-m}(\Gamma)$.

If $m>3$,
the graded piece $\Gamma^{1-m}$ is spanned by three types of elements:
\begin{enumerate}
\item \label{it:typ1}
$u_1 \eps_{k_1}^* u_2 \eps_{k_2}^* u_3 \dots \eps_{k_{m-1}}^* u_m$,
where $u_1, u_2, \dots, u_m$ are paths in $Q$ and $1 \leq k_j \leq n$
for each $1 \leq j \leq m-1$;

\item \label{it:typ2}
$u \eps_k v \eps_l^* w$ and $u \eps_k^* v \eps_l w$, where
$u, v, w$ are paths in $Q$ and $1 \leq k, l \leq n$;

\item \label{it:typ3}
$u \alpha^* v$, where $u, v$ are paths in $Q$ and $\alpha \in Q_1$.
\end{enumerate}
If $m=3$, in addition to the elements above there is a fourth type
\begin{enumerate}
\setcounter{enumi}{3}
\item \label{it:typ4}
$u \eps_k v \eps_l w$, where $u, v, w$ are paths in $Q$ and
$1 \leq k, l \leq n$.
\end{enumerate}

It suffices to prove that the differential of any of these elements belongs
to $E'$. This is clear for the elements of the
type~\eqref{it:typ1} since
\[
d(u_1 \eps_{k_1}^* u_2 \eps_{k_2}^* u_3 \dots \eps_{k_{m-1}}^* u_m) =
\sum_{j=1}^{m-1} \pm u_1 \eps_{k_1}^* u_2 \dots \eps_{k_{j-1}}^* u_j \rho_j
u_{j+1} \eps_{k_{j+1}}^* \dots \eps_{k_{m-1}}^* u_m
\]
and none of the arrows $\eps_1, \eps_2, \dots, \eps_n$ can appear in the right
hand side. This is also clear for the elements of type~\eqref{it:typ4} since
their differential vanishes.

Consider an element of type~\eqref{it:typ2}. Then
$d(u \eps_k v \eps_l^* w) = u \eps_k v \rho_l w$, hence the differential is
spanned by elements of the form $u' \eps_k v'$ where $u', v'$ are paths in $Q$
and $v'$ has positive length. The case of $d(u \eps_k^* v \eps_l w)$ is
similar.

Consider an element of type~\eqref{it:typ3}. Then
$d(u \alpha^* v) = u (\partial_{\alpha} W) v = \sum_{k=1}^n u 
\partial_\alpha (\eps_k \rho_k) v$.
Our assumption that $R \subseteq \fr^2$ implies that each
$\partial_\alpha (\eps_k \rho_k)$ is a linear combination of terms
$u'' \eps_k v''$ where at least one of the paths $u'', v''$ has positive
length.
\end{proof}

\begin{cor}
If $m>2$ and $R \subseteq \fr^2$, then $\hh^{2-m}(\Gamma)=0$ if and only if
$R$ is empty.
\end{cor}

Consider the triangulated category
$\cC_{(Q,R,m)} = \per \Gamma(Q,R,m) / \cD_{\fd}(\Gamma(Q,R,m))$.

\begin{theorem} \label{t:mCYtilted}
Let $Q$ be a quiver, let $R$ be a finite sequence of relations on $Q$ such
that the algebra $A=KQ/(R)$ is finite-dimensional and let $m>2$ be an integer.
Then:
\begin{enumerate}
\renewcommand{\theenumi}{\alph{enumi}}
\item
The category $\cC = \cC_{(Q,R,m)}$ is $\Hom$-finite and $m$-Calabi-Yau.

\item
The image $T$ of $\Gamma(Q,R,m)$ in $\cC$ is an $m$-cluster-tilting object
with
$\End_{\cC}(T) \cong A$.

\item
If $R \subseteq \fr^2$ then $\dim_K \Hom_{\cC}(T, \Sigma^{-(m-2)} T) \geq |R|$.
\end{enumerate}
\end{theorem}
\begin{proof}
Recall that $\Gamma(Q,R,m)= \Gamma_{m+1}(\wt{Q},W)$ where $(\wt{Q},W)$ is the
graded quiver with superpotential of degree $2-m$ associated to the triple
$(Q,R,m)$ as in Construction~\ref{con:QW}.
We claim that $(\wt{Q},W)$ satisfies the conditions of Theorem~\ref{t:mCY}.
Indeed, condition~\eqref{it:t:deg} holds since the degree of any arrow in
$\wt{Q}$ is either $0$ or $2-m$, and condition~\eqref{it:t:fd} holds since by
Lemma~\ref{l:H0}, $\hh^0(\Gamma) \cong A$ and the algebra $A$ is
assumed to be finite-dimensional.

Hence, by Theorem~\ref{t:mCY}, the triangulated category $\cC = \cC_{(Q,R,m)}
= \cC_{(\wt{Q},W)}$ is $\Hom$-finite, $m$-Calabi-Yau and the image $T$ of
$\Gamma = \Gamma(Q,R,m)$ under the canonical projection $\per \Gamma \to
\cC$ is an $m$-cluster-tilting object whose endomorphism algebra is
$\End_{\cC}(T) \cong \hh^0(\Gamma) \cong A$.
The last assertion follows from Lemma~\ref{l:snex} and Lemma~\ref{l:H2mR}.
\end{proof}

\begin{remark}
The $m$-Calabi-Yau category $\cC$ of Theorem~\ref{t:mCYtilted} depends on
$Q$ and $R$ and not only on the algebra $A$. There exist quivers $Q$ and
sequences of relations $R$ and $R'$ on $Q$ such that the algebras
$KQ/(R)$ and $KQ/(R')$ are finite-dimensional and isomorphic but for any
$m>2$ the categories $\cC_{(Q,R,m)}$ and $\cC_{(Q,R',m)}$ are not equivalent,
see Example~\ref{ex:nonuniq} below.
\end{remark}

\begin{remark} \label{rem:endproj}
Keep the notations and assumptions of Theorem~\ref{t:mCYtilted} and write $A$
as $A = \oplus_{i \in Q_0} P_i$ where $P_i = e_i A$ are the indecomposable
projective right $A$-modules. The decomposition $\Gamma = \oplus_{i \in Q_0}
e_i \Gamma$ for $\Gamma=\Gamma(Q,R,m)$ induces a decomposition
$T = \oplus_{i \in Q_0} T_i$ with $T_i$ being the image of $e_i \Gamma$ under
the canonical projection to $\cC$. Hence for any finitely generated
projective $A$-module 
$P = \oplus_{i \in Q_0} P_i^{e_i}$ such that $e_i>0$ for all $i \in Q_0$ 
there exists an $m$-cluster-tilting object $T_P = \oplus_{i \in Q_0} T_i^{e_i}$
in $\cC$ with $\End_{\cC}(T_P) \cong \End_A(P)$.
\end{remark}

The next result shows that the $m$-cluster category of any acyclic quiver
can be realized as a category of the form $\cC(Q,R,m)$. Recall that
a quiver is \emph{acyclic} if it has no cycles of positive length.
We refer to~\cite[Corollary~3.4]{Guo11} for a related result.

\begin{prop} \label{p:vosnex}
Let $Q$ be a quiver, let $R$ be a finite sequence of relations on $Q$ such
that $R \subseteq \fr^2$ and the algebra $KQ/(R)$ is finite-dimensional, and
let $m>2$ be an integer. 
Then the following conditions are equivalent, where $\cC$ denotes the 
$m$-Calabi-Yau category $\cC_{(Q,R,m)}$ and $T$ is the canonical
$m$-cluster-tilting object in $\cC$.
\begin{enumerate}
\renewcommand{\theenumi}{\alph{enumi}}
\item \label{it:p:KQ}
The quiver $Q$ is acyclic and the sequence $R$ is empty;

\item \label{it:p:B0}
The dg-algebra $B(Q,R)$ is concentrated in degree $0$ and has finite total
dimension;

\item \label{it:p:vosnex}
$\Hom_{\cC}(T, \Sigma^{-i} T) = 0$ for any $0 < i < m-1$;

\item \label{it:p:m2}
$\Hom_{\cC}(T, \Sigma^{-(m-2)} T) = 0$.
\end{enumerate}
Moreover, if any of these equivalent conditions holds and the field $K$ is
algebraically closed, then $\cC$ is triangle equivalent to the $m$-cluster
category of $Q$.
\end{prop}
\begin{proof}
The equivalence of~\eqref{it:p:KQ} and~\eqref{it:p:B0} is clear.
Let $\Gamma = \Gamma(Q,R,m)$.
For the implication \eqref{it:p:KQ} $\Rightarrow$ \eqref{it:p:vosnex},
note that if $Q$ is acyclic and $R$ is empty then
$\Hom_{\cC}(T, \Sigma^{-i} T) \cong \hh^{-i}(\Gamma) = 0$
for any $0 < i < m-1$ by Lemma~\ref{l:snex} and Lemma~\ref{l:Qempty}.
The implication \eqref{it:p:vosnex} $\Rightarrow$ \eqref{it:p:m2} is clear.
For the implication \eqref{it:p:m2} $\Rightarrow$ \eqref{it:p:KQ}, note that
by Lemma~\ref{l:snex} and Lemma~\ref{l:H2mR}
\[
\dim_K \Hom_{\cC}(T, \Sigma^{-(m-2)} T) = \dim_K \hh^{2-m}(\Gamma) \geq |R|,
\]
hence if $\Hom_{\cC}(T, \Sigma^{-(m-2)} T)$ vanishes $R$ must be empty and
then $Q$ is acyclic by our assumption that $KQ/(R)$ is finite-dimensional.

If any of these conditions holds, then $\cC$ is an algebraic $\Hom$-finite,
$m$-Calabi-Yau triangulated category with an $m$-cluster-tilting object $T$
such that $\End_{\cC}(T) \cong KQ$ for an acyclic quiver $Q$ and
$\Hom_{\cC}(T, \Sigma^{-i} T) = 0$ for any $0 < i < m-1$.
By the characterization of higher cluster categories of Keller and Reiten
\cite[Theorem~4.2]{KellerReiten08}) if $K$ is algebraically closed, then $\cC$
is triangle equivalent to the $m$-cluster-category of the quiver $Q$.
\end{proof}

\subsection{Systems of relations}

Let $Q$ be a quiver and let $\fr$ be the two-sided ideal of $KQ$ generated
by the arrows of $Q$.
An ideal $I$ of $KQ$ is \emph{admissible} if there exists some $N \geq 2$
such that $\fr^N \subseteq I \subseteq \fr^2$.

\begin{defn}[\protect{\cite{Bongartz83}}]
Let $I$ be an ideal of $KQ$.
A \emph{system of relations for $I$} is a set $R$ of relations such that
$R$, but no proper subset of it, generates $I$ as a two-sided ideal.
\end{defn}

The following statement is well-known.

\begin{lemma} \label{l:sysrel}
If $I$ is admissible then there exists a finite system of relations for $I$.
\end{lemma}
\begin{proof}
By assumption, there is some $N \geq 2$ such that $\fr^N \subseteq I$.
The algebra $KQ/\fr^N$ is finite-dimensional, as it is spanned by all the
paths of $Q$ of length smaller than $N$. Therefore the space $I/\fr^N$ is also
finite-dimensional.
For each $i,j \in Q_0$, choose a basis of $e_i (I/\fr^N) e_j$ and choose
a set $R_{i,j}$ inside $e_i I e_j$ whose image modulo $\fr^N$ equals that
basis. Then $I=(R)$ for the finite set $R$ given by
\[
R = \{\text{the paths of length $N$ in $Q$}\} \cup
\bigcup_{i,j \in Q_0} R_{i,j} .
\]

If $R$ is not a system of relations for $I$ then there is a proper subset $R'$
of $R$ such that $I=(R')$. In this way we can repeatedly remove elements and
still have a set generating $I$. Since $R$ is finite, this process must
terminate and we eventually end with a system of relations for $I$.
\end{proof}

Under some conditions, another approach to the construction of systems of
relations involves lifting of basis elements of the space $I/(I \fr + \fr I)$,
see the discussion in~\cite[\S7]{BIKR08}.

\begin{lemma} \label{l:fdIrrI}
If $I$ is admissible, then $I/(I \fr + \fr I)$ is finite-dimensional.
\end{lemma}
\begin{proof}
Let $N \geq 2$ be such that $\fr^N \subseteq I$. Then 
$\fr^{N+1} \subseteq (I \fr + \fr I)$ and we have an inclusion and
a surjection
\[
KQ/\fr^{N+1} \supset I/\fr^{N+1} \twoheadrightarrow I/(I \fr + \fr I).
\]
The claim now follows since the quotient $KQ/\fr^{N+1}$ is finite-dimensional.
\end{proof}

\begin{lemma} \label{l:minrel}
Let $I$ be an ideal of $KQ$ and let $R$ be a set of relations inside $I$.
\begin{enumerate}
\renewcommand{\theenumi}{\alph{enumi}}
\item \label{it:l:span}
If $I=(R)$ then the image of $R$ modulo $I \fr + \fr I$ spans the vector space
$I/(I \fr + \fr I)$.

\item \label{it:l:lift}
Assume that the ideal (R) is admissible and the image of $R$ modulo
$I \fr + \fr I$ spans the vector space $I/(I \fr + \fr I)$. Then $I=(R)$.

\item \label{it:l:liftbasis}
Assume that the ideal (R) is admissible and the image of $R$ modulo
$I \fr + \fr I$ is a basis of the vector space $I/(I \fr + \fr I)$.
Then $R$ is a system of relations for $I$.
\end{enumerate}
\end{lemma}
\begin{proof}
For part~\eqref{it:l:span}, observe that since each $\rho \in R$ is a relation,
multiplying it from the left or from the right by an element of the form
$\sum_{i \in Q_0} \lambda_i e_i$ gives a scalar multiple of $\rho$.

For part~\eqref{it:l:lift}, we slightly modify the argument
in~\cite[Lemma~3.6]{BMR06}. Let $N \geq 2$ such that $\fr^N \subseteq (R)$
and let $x \in I$. By assumption, we can write
\begin{equation} \label{e:IbyR}
x = \lambda_1 \rho_1 + \dots + \lambda_n \rho_n + x'_1 r'_1 + \dots +
x'_k r'_k + r''_1 x''_1 + \dots + r''_{\ell} x''_{\ell}
\end{equation}
with scalars $\lambda_1, \dots, \lambda_n \in K$, 
relations $\rho_1, \dots, \rho_n \in R$ and
$x'_1, \dots, x'_k, x''_1, \dots, x''_{\ell} \in I$, 
$r'_1, \dots, r'_k, r''_1, \dots, r''_{\ell} \in \fr$.
Writing each of the elements $x'_1, \dots, x'_k, x''_1, \dots, x''_{\ell}$
using~\eqref{e:IbyR} and repeating this process $N$ times, we conclude that
$x = \rho + r$ where $\rho \in (R)$ and $r \in \fr^N$, hence $x \in (R)$.

Finally, part~\eqref{it:l:liftbasis} follows from~\eqref{it:l:span}
and~\eqref{it:l:lift}. Moreover, $R$ is finite by Lemma~\ref{l:fdIrrI}.
\end{proof}

\begin{example} \label{ex:minrel}
The assumption in parts~\eqref{it:l:lift} and~\eqref{it:l:liftbasis} that the
ideal $(R)$ is admissible cannot be dropped. For example, consider the quiver
$Q$ given by
\[
\xymatrix{
{\bullet} \ar@(ul,dl)[]_{\alpha} \ar@(ur,dr)[]^{\beta}
}
\]
and let
$I=(\alpha^2-\beta \alpha \beta, \beta^2 - \alpha \beta \alpha, \alpha^2 \beta)$.
One can check that $\fr^5 \subseteq I \subseteq \fr^2$ hence the ideal $I$ is
admissible. The $8$-dimensional algebra $KQ/I$ is an algebra of quaternion type
in the sense of Erdmann~\cite{Erdmann90}. When $K$ is algebraically closed of
characteristic $2$, this algebra is isomorphic to the group algebra of the
quaternion group.

The image of $\alpha^2 \beta$ in $I \fr + \fr I$ vanishes, as the following
calculation shows:
\[
\alpha^2 \beta = (\alpha^2 - \beta \alpha \beta) \beta + \beta \alpha
(\beta^2 - \alpha \beta \alpha) + \beta \alpha^2 \beta \alpha \in I \fr + \fr I,
\]
hence $I/(I \fr + \fr I)$ is spanned by the images of the elements
$\alpha^2 - \beta \alpha \beta$ and $\beta^2 - \alpha \beta \alpha$. 
Nevertheless,
the ideal $I' = (\alpha^2-\beta \alpha \beta, \beta^2 - \alpha \beta \alpha)$
is not equal to $I$. Indeed, by letting $\alpha$ and $\beta$ act
on $K$ as the identity we get a one-dimensional module over $KQ/I'$ with a
non-zero action of $\alpha^2 \beta$, hence $\alpha^2 \beta \not \in I'$.
\end{example}

\subsection{Finite-dimensional algebras are $(m>2)$-CY-tilted}
In this section we assume that the field $K$ is algebraically closed.
A finite-dimensional algebra $A$ over $K$ is called \emph{basic} if
$A_A \cong P_1 \oplus \dots \oplus P_r$ where $P_1, \dots, P_r$ are
representatives of the isomorphism classes of the indecomposable projective
right $A$-modules.

\begin{theorem}[Gabriel] \label{t:Gabriel}
Let $A$ be a basic, finite-dimensional algebra over $K$. Then
there exist a quiver $Q$ and an admissible ideal $I$ of $KQ$ such
that $A \cong KQ/I$.
\end{theorem}
\begin{proof}
See~\cite[\S4.3]{Gabriel80}.
\end{proof}

Let $A$ be a basic, finite-dimensional algebra.
By Theorem~\ref{t:Gabriel}, we can write $A=KQ/I$ for a quiver $Q$ and an
admissible ideal $I$ of $Q$. We denote by $S_i$ the simple $A$-module
corresponding to a vertex $i \in Q_0$ and consider the $A$-module
$S = \bigoplus_{i \in Q_0} S_i$.

\begin{lemma} \label{l:mincardR}
If $R$ is a system of relations for $I$ then
$|R| \geq \dim_K \Ext^2_A(S,S)$.
\end{lemma}
\begin{proof}
We have $|R| \geq \dim_K I/(I \fr + \fr I) = \dim_K \Ext^2_A(S,S)$ where the
left inequality is a consequence of Lemma~\ref{l:minrel}\eqref{it:l:span} and
the right equality is~\cite[Corollary~1.1]{Bongartz83}.
\end{proof}

Note that Example~\ref{ex:minrel} shows that the inequality in
Lemma~\ref{l:mincardR} can be strict.

\begin{theorem} \label{t:mCYbasic}
Let $A$ be a basic finite-dimensional algebra.
Then the set of pairs $(Q,R)$ consisting of a quiver $Q$ and a sequence of relations $R \subseteq \fr^2$ such that $A \cong KQ/(R)$ is not empty.
For any such pair $(Q,R)$ and any integer $m>2$, 
the triangulated category $\cC=\cC_{(Q,R,m)}$ is
$\Hom$-finite, $m$-Calabi-Yau and its canonical $m$-cluster-tilting object $T$
satisfies $\End_{\cC}(T) \cong A$ and
$\dim_K \Hom_{\cC}(T, \Sigma^{-(m-2)}T) \geq \dim_K \Ext^2_A(S,S)$.
\end{theorem}
\begin{proof}
The first claim is a consequence of Theorem~\ref{t:Gabriel} and
Lemma~\ref{l:sysrel}. The second claim is a consequence of 
Theorem~\ref{t:mCYtilted} and Lemma~\ref{l:mincardR}.
\end{proof}

\begin{cor}
A finite-dimensional algebra over an algebraically closed field is
$m$-CY-tilted for any $m>2$.
\end{cor}
\begin{proof}
We can write $A \cong \End_{\bar{A}}(P)$ for a finite-dimensional, basic
algebra $\bar{A}$ and a finitely generated projective $\bar{A}$-module $P$
containing as direct summands all the indecomposable projective
$\bar{A}$-modules. Hence the claim is a consequence of Theorem~\ref{t:mCYbasic}
and Remark~\ref{rem:endproj}.
\end{proof}

\begin{remark}
If $\gldim A \geq 2$, then for any of the $m$-Calabi-Yau categories $\cC$
of Theorem~\ref{t:mCYbasic}, the small negative extension
$\Hom_{\cC}(T, \Sigma^{-(m-2)}T)$ of the canonical $m$-cluster-tilting object
$T$ cannot vanish. Had it vanished, Proposition~\ref{p:vosnex} would then imply
that $A$ is the path algebra of an acyclic quiver, a contradiction.
\end{remark}

\subsection{Examples}

Our first example is similar in spirit to~\cite[Example~3.5]{Guo11}.

\begin{example}
Consider the algebra $A=KQ/I$ where $Q$ is the left quiver
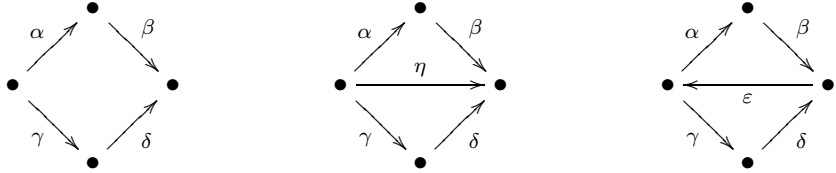
\begin{align*}
\xymatrix@=1.5pc{
& {\bullet} \ar[dr]^{\beta} \\
{\bullet} \ar[ur]^{\alpha} \ar[dr]_{\gamma} && {\bullet} \\
& {\bullet} \ar[ur]_{\delta}
}
& &
\xymatrix@=1.5pc{
& {\bullet} \ar[dr]^{\beta} \\
{\bullet} \ar[ur]^{\alpha} \ar[dr]_{\gamma} \ar[rr]^{\eta} && {\bullet} \\
& {\bullet} \ar[ur]_{\delta}
}
& &
\xymatrix@=1.5pc{
& {\bullet} \ar[dr]^{\beta} \\
{\bullet} \ar[ur]^{\alpha} \ar[dr]_{\gamma} && {\bullet} \ar[ll]^{\eps} \\
& {\bullet} \ar[ur]_{\delta}
}
\end{align*}
and $I=(R)$ for the system of relations $R=\{\alpha \beta - \gamma \delta\}$.
The algebra $A$ has global dimension $2$, hence it cannot be 2-CY-tilted
by~\cite[Corollary~2.1]{KellerReiten07}.

Let $B=B(Q,R)$ be the dg-algebra of Construction~\ref{con:dgQR}.
Its graded quiver is shown in the middle; the arrows $\alpha, \beta, \gamma$ and 
$\delta$ have degree $0$ while $\eta$ has degree $-1$ and $d(\eta)=\alpha \beta
- \gamma \delta$. Observe that the dg-algebra $B$ is quasi-isomorphic to
$A \cong \hh^0(B)$.

Let $m \geq 2$ and let $(\wt{Q},W)$ be the graded quiver with homogeneous
superpotential of degree $2-m$ of Construction~\ref{con:QW}.
The graded quiver $\wt{Q}$ is shown on the right; the degrees
of the arrows $\alpha, \beta, \gamma, \delta$ are $0$, that of $\eps$
is $2-m$ and the superpotential is $W=\eps(\alpha \beta - \gamma \delta)$.

The Ginzburg dg-algebra $\Gamma = \Gamma_{m+1}(\wt{Q},W)$ is the
$(m+1)$-Calabi-Yau completion of $B$ (Lemma~\ref{l:preprojB}).
Since $B$ is quasi-isomorphic to $A$
and $A$ is derived equivalent to the path algebra of the Dynkin quiver $D_4$,
the Morita invariance of Calabi-Yau
completions~\cite[Proposition~4.2]{Keller11} implies that $\Gamma$ is
Morita equivalent to $\Gamma_{m+1}(D_4,0)$, hence by~\cite[Corollary~3.4]{Guo11}
(or Proposition~\ref{p:vosnex}) the generalized $m$-cluster category
$\cC_{\Gamma}$ is triangle equivalent to the $m$-cluster category of type $D_4$.

If $m=2$ then $\hh^0(\Gamma) \cong K\wt{Q}/(\wt{R})$, where
$\wt{R}=\{\alpha \beta - \gamma \delta, \eps \alpha, \beta \eps, \eps \gamma,
\delta \eps\}$. This is a cluster-tilted algebra~\cite{BMR06} of type $D_4$,
which is the relation-extension~\cite{ABS08} of the tilted algebra $A$.
If $m>2$ then $\hh^0(\Gamma) \cong A$ and $A$ is $m$-CY-tilted by
Theorem~\ref{t:mCYtilted}.
\end{example}

\begin{example} \label{ex:nonuniq}
Let $K$ be algebraically closed and consider the algebra $A=K$ whose quiver
$Q$ is $\bullet$. Consider the two sequences of relations
$R=\{\}$ and $R'=\{0\}$ on $Q$.

Let $m > 2$. The Ginzburg dg-algebras $\Gamma=\Gamma(Q,R,m)$ and
$\Gamma'=\Gamma(Q,R',m)$ are given by the following graded quivers with
differentials
\begin{align*}
\begin{array}{lc}
\begin{array}{l}
_{|t|=-m} \\
_{d(t)=0}
\end{array}
&
\begin{array}{c}
\xymatrix{
{\bullet} \ar@(ur,dr)[]^t
}
\end{array}
\end{array}
&&
\begin{array}{cl}
\begin{array}{c}
\xymatrix{
{\bullet} \ar@(ur,dr)[]^t \ar@(l,u)[]^{\eps^*} \ar@(l,d)[]_{\eps}
}
\end{array}
&
\begin{array}{l}
_{|\eps^*|=-1,\, |\eps|=-(m-2),\, |t|=-m} \\
_{d(\eps) = d(\eps^*) = 0} \\
_{d(t) = \eps \eps^* - (-1)^m \eps^* \eps}
\end{array}
\end{array}
\end{align*}

A basis for the graded piece $\Gamma'^i$ is given by
$(\eps^*)^i$ if $0 < i < m-2$; $(\eps^*)^{m-2}, \eps$ if $i = m-2$;
$(\eps^*)^{m-1}, \eps \eps^*, \eps^* \eps$ if $i=m-1$ and $m>3$; and
$(\eps^*)^{m-1}, \eps \eps^*, \eps^* \eps, \eps^2$ if $i=m-1$ and $m=3$,
hence one computes
\begin{align} \label{e:hiGamma}
\dim_K \hh^{-i}(\Gamma) = 
\begin{cases}
1 & \text{if $i=0$,} \\
0 & \text{if $0 < i < m$,}
\end{cases}
&&
\dim_K \hh^{-i}(\Gamma') =
\begin{cases}
1 & \text{if $0 \leq i < m-2$,} \\
2 & \text{if $i=m-2$,} \\
2 + \delta_{3,m} & \text{if $i=m-1$.}
\end{cases}
\end{align}

Let $\cC=\cC(Q,R,m)$ and $\cC'=\cC(Q,R',m')$ be the corresponding
$m$-Calabi-Yau categories and let $T \in \cC$ and $T' \in \cC'$ be the
canonical $m$-cluster-tilting objects.
We have $\End_{\cC}(T) \cong \End_{\cC'}(T') \cong K$.

By Proposition~\ref{p:vosnex}, the category $\cC$ is triangle equivalent to the
$m$-cluster category of $Q$. Hence the $m$-cluster-tilting objects in $\cC$ 
are the indecomposable objects $\Sigma^j T$ for $0 \leq j < m$.
Since $\Hom_{\cC}(T, \Sigma^{-i} T) = 0$ for any $0 < i < m$, this remains
true if we replace $T$ by any of the other $m$-cluster-tilting objects
$\Sigma^j T$ in $\cC$.
On the other hand, $T'$ is an $m$-cluster-tilting object in $\cC'$ with 
$\Hom_{\cC'}(T', \Sigma^{-i} T') \neq 0$ for any $0 < i < m$
by~\eqref{e:hiGamma},
hence $\cC'$ cannot be triangle equivalent to $\cC$.
\end{example}

The previous example can be generalized as follows.
A \emph{Dynkin quiver} is a quiver obtained by orienting the edges of a Dynkin
diagram of type $A_n$ ($n \geq 1$), $D_n$ ($n \geq 4$) or $E_n$ ($n=6,7,8$).

\begin{prop}
Let $Q$ be a Dynkin quiver and let $m > 2$.
There exists an $m$-Calabi-Yau triangulated category
$\cC'$ with an $m$-cluster-tilting object $T'$ such that
$\End_{\cC'}(T') \cong KQ$ but $\cC'$ is not triangle equivalent to the
$m$-cluster category of $Q$.
\end{prop}
\begin{proof}
Let $\cC$ be the $m$-cluster category of $Q$. Since $Q$ is Dynkin,
the category $\cC$ has only finitely many indecomposable objects, hence the
number of $m$-cluster-tilting objects $T$ in $\cC$ such that $\End_{\cC}(T)
\cong KQ$ is finite. Therefore there exists an integer $n$ such that
$\dim_K \Hom_{\cC}(T, \Sigma^{-(m-2)} T) < n$ for any such $T$.

Let $R = \{0, 0, \dots, 0\}$ be a sequence consisting of $n$ zero elements
(it does not matter which starting and ending vertex we assign to each zero
element) and let $\cC' = \cC_{(Q,R,m)}$. By Theorem~\ref{t:mCYtilted}, the
category $\cC'$ is $\Hom$-finite and $m$-Calabi-Yau with an $m$-cluster-tilting 
object $T'$ satisfying $\End_{\cC'}(T') \cong KQ$ and
$\dim_K \Hom_{\cC'}(T', \Sigma^{-(m-2)} T') \geq n$.

If $F \colon \cC' \simeq \cC$ were a triangulated equivalence, then $T=FT'$
would be an $m$-cluster-tilting object in $\cC$ with $\End_{\cC}(T) \cong KQ$
and $\dim_K \Hom_{\cC}(T, \Sigma^{-(m-2)}T) \geq n$, a contradiction.
\end{proof}

\bibliographystyle{amsplain}
\bibliography{dCYtilt}

\end{document}